\documentclass[12pt, reqno]{amsart}
\usepackage{amsmath, amsthm, amscd, amsfonts, amssymb, graphicx, color}
\usepackage[bookmarksnumbered, colorlinks, plainpages]{hyperref}

\newtheorem{theorem}{Theorem}[section]
\newtheorem{lemma}[theorem]{Lemma}
\newtheorem{proposition}[theorem]{Proposition}

\theoremstyle{definition}
\newtheorem{definition}[theorem]{Definition}

\theoremstyle{remark}
\newtheorem{remark}[theorem]{Remark}
\numberwithin{equation}{section}

\newcommand\N{\mathbb{N}}
\newcommand\R{\mathbb{R}}

\newcommand\Z{\mathbb{Z}}

\def\sideremark#1{\ifvmode\leavevmode\fi\vadjust{\vbox
to0pt{\vss \hbox to 0pt{\hskip\hsize\hskip1em
\vbox{\hsize2cm\tiny\raggedright\pretolerance10000
\noindent#1\hfill}\hss}\vbox to8pt{\vfil}\vss}}}

\begin{document}
\setcounter{page}{1}

\title[Upper Beurling Density of Translates in $L^p(\R^d)$]{Upper Beurling Density
of Systems formed by Translates of finite Sets of Elements in $L^p(\R^d)$}

\author[B. Liu, R. Liu]{Bei Liu$^1$ and Rui Liu$^2$$^{*}$}

\address{$^{1}$ Department of Mathematics, Tianjin University of Technology, Tianjin
300384, P.R. China.}
\email{\textcolor[rgb]{0.00,0.00,0.84}{beiliu1101@gmail.com}}

\address{$^{2}$ Department of Mathematics and LPMC, Nankai University,
Tianjin 300071, P.R. China.}
\email{\textcolor[rgb]{0.00,0.00,0.84}{ruiliu@nankai.edu.cn}}




\thanks{$^{*}$ Corresponding author.
\newline \indent This work was partially done while the second author was visiting the University of Texas at Austin and Texas A\&M University, and in the Linear Analysis Workshop at Texas A\&M University which was funded by NSF Workshop Grant. Both authors are also supported by the National Natural Science Foundation of China 11001134 and 11126250, the Fundamental Research Funds for the
Central Universities, and the Tianjin
Science \& Technology Fund 20100820.}

\begin{abstract}
In this paper, we prove that if a finite disjoint union of
translates $\bigcup_{k=1}^n\{f_k(x-\gamma)\}_{\gamma\in\Gamma_k}$ in
$L^p(\R^d)$ $(1<p<\infty)$ is a $p'$-Bessel sequence for some $1<
p'<\infty$, then the disjoint union $\Gamma=\bigcup_{k=1}^n
\Gamma_k$ has finite upper Beurling density, and that if
$\bigcup_{k=1}^n\{f_k(x-\gamma)\}_{\gamma\in\Gamma_k}$ is a
$(C_q)$-system with $1/p+1/q=1$, then $\Gamma$ has infinite upper
Beurling density. Thus, no finite disjoint union of translates in
$L^p(\R^d)$
can form a 
$p'$-Bessel $(C_q)$-system for any $1< p'<\infty$. Furthermore, by using
techniques from the geometry of Banach spaces, we obtain that, for
$1<p\le2$, no finite disjoint union of translates in $L^p(\R^d)$ can
form an unconditional basis.
\end{abstract} \maketitle

\section{Introduction}


Given $1<p<\infty$ and $\gamma\in\R^d$, we define the translation
operator $T_\gamma$ on $L^p(\R^d)$ by $(T_\gamma f)(x)=f(x-\gamma)$
for all $x\in\R^d.$ If $\Gamma\subset\R^d,$ then the collection of
translations of $f\in L^p(\R^d)$ along $\Gamma$ is defined to be
$T_p(f,\Gamma)=\{T_\gamma f\}_{\gamma\in\Gamma}.$ Our main focus
shall be on the upper Beurling density of such $\Gamma$, the
disjoint union $\bigcup_{k=1}^n\Gamma_k$, given that
$\bigcup_{k=1}^n T_p(f_k,\Gamma_k)$ has some additional structure in
$L^p(\R^d).$ The ``additional structure'' takes two forms:
$\bigcup_{k=1}^n T_p(f_k,\Gamma_k)$ is a $p'$-Bessel sequence or is
a $(C_q)$-system.

The nature of $T_p(f,\Gamma)$ has been studied in a number of
papers \cite{Wi,AOl,ER,Ol}, mainly using techniques of harmonic
analysis. Our techniques will come partially from the geometry of
Banach spaces. Recall that, in 1992, Olson and Zalik \cite{OZ}
proved that there do not exist any Riesz bases for $L^2(\R)$
generated by $T_2(f,\Gamma)$. Then Christensen \cite{Ch1}
conjectured that there are no frames for $L^2(\R)$ of the form
$\bigcup^n_{k=1} T_2(f_k,\Gamma_k)$. In 1999, Christensen, Deng and
Heil \cite{CDH} proved this conjecture by studying density of
frames. For more density theorems, please see the research survey
\cite{He}. Recently, Odell, Sari, Schlumprecht and Zheng \cite{OSSZ}
used techniques largely from the geometry of Banach spaces to consider
the closed subspace of $L^p(\R)$ generated by translates of one
element in $L^p(\R)$.
%

In Section 2, we extend the concept of $(C_q)$-system from Hilbert
spaces to reflexive Banach spaces and give our basic Lemma
\ref{lem:5} and examples in $L^p(\R^d)$.

In section 3, by using techniques in \cite{CDH,OSSZ}, we prove that
if $\bigcup_{k=1}^n T_p(f_k,\Gamma_k)$ in $L^p(\R^d)$ $(1<p<\infty)$
is a $p'$-Bessel sequence for some $1< p'<\infty$, then the disjoint
union $\Gamma=\bigcup_{k=1}^n \Gamma_k$ has finite upper Beurling
density, and that if $\bigcup_{k=1}^n T_p(f_k,\Gamma_k)$ is a
$(C_q)$-system with $1/p+1/q=1$, then $\Gamma$ has infinite upper
Beurling density.
Thus, no collection $\bigcup_{k=1}^n T_p(f_k,\Gamma_k)$ of pure
translates can form a $p'$-Bessel $(C_q)$-system in $L^p(\R^d)$ for
any $1< p'<\infty$. This extends the Christensen/Deng/Heil density
result in \cite{CDH} from classical (Hilbert) frames in $L^2(\R^d)$
to more general $p'$-Bessel $(C_q)$-systems in $L^p(\R^d)$.

In the last section. by using techniques from the geometry of Banach spaces,
we obtain that there is no unconditional basis of $L^p(\R^d)$ of
the form $\bigcup_{k=1}^n T_p(f_k,\Gamma_k)$ for $1<p\le 2$. It
partially extends the latest work \cite{OSSZ} on uniformly separated
translates of one element in $L^p(\R)$. The extension is to higher
dimensions, to multiple generating functions, and to completely
arbitrary sets of translates.
%
%
\section{Preliminaries and notation}
%

In 2001, Aldroubi, Sun and Tang \cite{AST} introduced the concept of
$p$-frame in $L^p(\R)$, which is a generalization of classical
(Hilbert) frames \cite{DS,Ca,Ch2} and can be naturally extended to
Banach spaces \cite{CS,CCS}.
\begin{definition}
Let $X$ be a separable Banach space and $1< p<\infty$. A family
$\{f_k\}_{k=1}^{\infty}\subset X$ is a \emph{$p$-frame} for $X^*$ if
there exist constants $A,B>0$ such that
$$
A\|h\|^p\le \sum_{k=1}^{\infty} |\langle h, f_k\rangle|^p\le
B\|h\|^p \ \ \mbox{ for all }\, h\in X^*.
$$
The number $A$ and $B$ are called the lower and upper p-frame
bounds. The sentence $\{f_k\}_{k=1}^{\infty}$ is a \emph{$p$-Bessel sequence} if the
right-hand side inequality holds. We say that
$\{f_k\}_{k=1}^{\infty}$ is a \emph{Bessel sequence} if it is a
$2$-Bessel sequence.
\end{definition}
In 2007, S. Nitzan and A. Olevskii introduced the concept of
$(C_q)$-system in Hilbert spaces \cite{NO1,NO2,NO3}. It is a weaker
form of the frame-type condition, which is a relaxed version of this
inequality:
\[A\|f\|^{2}\leq\sum^{\infty}_{k=1}|\langle f,f_k\rangle|^{2}\ \ \ \mbox{ for all }\, f\in\mathcal{H}.\]
Now we extend this useful definition of $(C_q)$-system to reflexive
Banach spaces.
\begin{definition}\label{def:1}
Let $X$ be a separable reflexive Banach space and $1<q<\infty$ be a
fixed number. We say that a sequence of
$\{f_k\}_{k=1}^{\infty}\subset X$ is a \emph{$(C_q)$-system in $X$
with constant $C>0$ (complete with $\ell_q$ control over the
coefficients)} if for every $f\in X$ and $\varepsilon>0,$ there exists a
linear combination $g=\sum a_kf_k$ such that
\begin{equation}\label{eq:5}\|f-g\|_X<\varepsilon\ \ \ \ \mbox{and}\ \ \ \Big(\sum|a_n|^q\Big)^{1/q}\le C\|f\|_X,
\end{equation}
where $C=C(q)$ is a positive constant not depending on $f$.
\end{definition}
\begin{remark}By Proposition \ref{th:6} in Section 4, given $p,q\in(1,\infty)$
with $1/p+1/q=1$, we have that: if $1<p\le 2$, then every
seminormalized unconditional basis of $L^p(\R^d)$ is a $q$-Bessel
$(C_2)$-system; if $2\le p<\infty$, then every seminormalized
unconditional basis of $L^p(\R^d)$ is a Bessel $(C_p)$-system.
\end{remark}
We define some types of sequences in $\R^d$ and upper Beurling
density \cite{CDH,Ch2}.


\begin{definition}
Let $\Gamma=\{\gamma_i\}_{i\in I}\subset\R^d.$
\begin{enumerate}
\item[$(i)$] A point $\gamma\in\R^d$ is an \emph{accumulation point} for $\Gamma$
if every open ball in $\R^d$ centered at $\gamma$ contains
infinitely many $\gamma_k.$
\item[$(ii)$] $\Gamma$ is \emph{$\delta$-uniformly separated} if $\delta=\inf_{i\neq j}|\gamma_i-\gamma_j|>0.$
The number $\delta$ is the \emph{separation constant}.
\item[$(iii)$] $\Gamma$ is \emph{relatively uniformly separated} if it is a finite
union of uniformly separated sequences $\Gamma_k$. That is to say
that $I$ can be partitioned into finite disjoint sets $I_1,...,I_n$
such that each sequence $\Gamma_k=\{\gamma_i\}_{i\in I_k}$ is
$\delta_k$-uniformly separated for some $\delta_k>0.$
\end{enumerate}
\end{definition}
For $h>0$ and $x\in\R^d,$ we define cube $Q_{h}(x)$ by
\[Q_{h}(x)=\prod^{d}_{i=1}[x_i-h/2,x_i+h/2), \ \ \mbox{ where }
x=(x_1, ... ,x_d).\] When $x=0,$ we use $Q_{h}$ instead of
$Q_{h}(0)$ for simplicity.
Let $\Gamma=\{\gamma_i\}_{i\in I}\subset\R^d$. For each $h>0$, let
$\nu_\Gamma^+(h)$ denote the largest number of points from $\Gamma$
that lie in any cube $Q_h(x)$, i.e.,
$$\nu_\Gamma^+(h)=\displaystyle{\sup_{x\in\R^d}\#(\Gamma\cap Q_h(x))}.$$ The
\emph{upper Beurling density} of $\Gamma$ is defined by
$$D^+(\Gamma)=\limsup_{h\rightarrow\infty}\frac{\nu_\Gamma^+(h)}{h^d}.$$

\begin{lemma}\label{lemma:1}
Let $\Gamma=\{\gamma_i\}_{i\in I}$ be a sequence in $\R^d$. Then the
following statements are equivalent.
\begin{enumerate}
\item[$(i)$] $D^+(\Gamma)<\infty.$
\item[$(ii)$] $\Gamma$ is relatively uniformly separated.
\item[$(iii)$] For some (and therefore every) $h>0$, there is a natural
number $N_h$ such that each cube $Q_h(hn)$, $n\in\Z^d$, contains at
most $N_h$ points from $\Lambda$. That is,
$$N_h=\sup_{n\in\Z^d}\#(\Lambda\cap Q_h(hn))<\infty.$$
\end{enumerate}
\end{lemma}

\section{Main results}


First we need the following basic lemma.
\begin{lemma}\label{lemma:2} Let $\Gamma$ be a sequence in $\R^d$, and $1<p,q<\infty$
with $1/p+1/q=1$. Assume that $f\in L^p(\R^d)$, $\tilde{f}\in
L^q(\R^d)$ and $\langle f, \tilde{f}\rangle\neq 0$. If $\Gamma$ is
not relatively uniformly separated, then, for any
$0<\varepsilon<|\langle f, \tilde{f}\rangle|$, we have
$$\sup_{\beta\in\R^d}\#\{\gamma\in \Gamma:|\langle T_\gamma f, T_\beta \tilde{f}\rangle|
>\varepsilon\}=\infty.$$
\end{lemma}
\begin{proof}
Consider the function $x \mapsto \langle T_x f, \tilde{f} \rangle$
for all $x\in\R^d.$ Since the function is continuous, for any
$0<\varepsilon<|\langle f, \tilde{f}\rangle|$ there is a cube $Q_h$
for some $h>0$ such that
$$\inf_{x\in Q_h} |\langle T_x f, \tilde{f}\rangle|
>\varepsilon.$$
Consider an arbitrary $N\in\N$, by Lemma \ref{lemma:1}, there is a
cube $Q_h(\beta)$ for some $\beta\in\R^d$, which contains at least
$N$ elements from $\Gamma$. Then for any $\gamma\in Q_h(\beta)$,
$\gamma-\beta\in Q_h$, we have $$|\langle T_\gamma f, T_\beta
\tilde{f}\rangle|= |\langle T_{\gamma-\beta} f,
\tilde{f}\rangle|>\varepsilon.$$ It follows that
$$\#\big\{\gamma\in \Gamma:|\langle T_\gamma f, T_\beta
\tilde{f}\rangle|>\varepsilon\big\}\ge \#\big(\Gamma\cap
Q_h(\beta)\big)\ge N.$$ Since $N\in\N$ is arbitrary, the conclusion
follows.
\end{proof}

For translate of one element, we get the following result.
\begin{proposition}\label{pp:1}
Let $1< p<\infty$, $f$ be a nonzero function in $L^p(\R^d)$, and
$\Gamma$ be a sequence in $\R^d$. If $T_p(f,\Gamma)$ is a
$p'$-Bessel sequence for some $1< p' <\infty$, then $\Gamma$ is
relatively uniformly separated.
\end{proposition}
\begin{proof}
Assume that $\Gamma$ is not relatively uniformly separated. Then for
any $N\in\N$, choose $\varepsilon$ such that $0<\varepsilon<\|f\|_p$.
By Hahn-Banach Theorem, there is an $\tilde{f}\in L^q(\R^d)$ with
$\|\tilde{f}\|_q=1$ such that $\langle
f,\tilde{f}\rangle=\|f\|_p>\varepsilon$. Then, by Lemma
\ref{lemma:2}, there exists $\beta\in\R^d$ such that
$$\#\big\{\gamma\in\Gamma:|\langle T_\gamma f, T_\beta
\tilde{f}\rangle|> \varepsilon\big\}\ge N.$$ Let
$\Gamma_N=\{\gamma\in\Gamma:|\langle T_\gamma f, T_\beta
\tilde{f}\rangle|>\varepsilon\}$. Then, we have
\begin{eqnarray*}
\sum_{\gamma\in\Gamma}|\langle T_\gamma f, T_\beta
\tilde{f}\rangle|^{p'} &\ge&\sum_{\gamma\in\Gamma_N}|\langle
T_\gamma f, T_\beta \tilde{f}\rangle|^{p'}
>N \varepsilon^{p'}.
\end{eqnarray*}
Since $N\in\N$ is arbitrary and
$\|T_\beta\tilde{f}\|_q=\|\tilde{f}\|_q$ is fixed, $T_p(f,\Gamma)$
is not a $p'$-Bessel sequence, which leads to a contradiction. Thus
$\Gamma$ is relatively uniformly separated.
\end{proof}

The following equivalent form extends Lemma 1 in \cite{NO1} by using
a standard duality argument in Banach spaces.
\begin{lemma}\label{lem:5}
Let $X$ be a separable reflexive Banach space and $1<p,q<\infty$
with $1/p+1/q=1$. A system $\{f_n\}\subset X$ is a $(C_q)$-system in
$X$ with constant $K>0$ if and only if
$$\frac{1}{K}\|h\|\le \Big(\sum_{n=1}^\infty |\langle h,f_n\rangle|^p\Big)^{1/p} \ \ \mbox{ for all }\, h\in X^*.$$
\end{lemma}
\begin{proof} For sufficiency, suppose that $\{f_n\}$ is not a $(C_q)$-system in $X$ with
constant $K>0$. Let
\[A:=\Big\{g=\sum a_nf_n: \big(\sum|a_n|^q\big)^{1/q}\le K\Big\}\]
be the set of finite linear combination and $\mathcal{C}$ be the
closure of $A$ in $X.$ It is easy to prove that $\mathcal{C}$ is a
closed convex subset of $X.$ By assumption, $\mathcal{C}$ does not
contain the closed unit ball $B$ of $X$. That is, there exists an $f\in
X$ with $\|f\|\le 1$, and $f$ is not in $\mathcal{C}$. By the
Hahn-Banach theorem, there is an $h\in X^*$ such that $|\langle h,
f\rangle|=1$ and $\sup_{g\in \mathcal{C}} |\langle h, g\rangle|<1.$
Hence, for sufficiently small $\varepsilon>0,$ we have $\sup_{g\in
\mathcal{C}} |\langle h, g\rangle|<1-\varepsilon.$ This implies for
any $M\in\mathbb{N},$
\begin{eqnarray*}
\Big(\sum_{n=1}^{M} |\langle h,f_n\rangle|^p\Big)^{1/p}
&=&\sup_{(\sum_{n=1}^{M} |\alpha_n|^q)^{1/q}\leq 1}
|\sum_{n=1}^{M}\langle h,
f_n\rangle\alpha_n|\\
&=&\frac{1}{K}\sup_{(\sum_{n=1}^{M} |\alpha_n|^q)^{1/q}\leq K}
|\langle h, \sum_{n=1}^{M}\alpha_nf_n\rangle|\\
&=&\frac{1}{K}\sup_{g\in \mathcal{C}} |\langle h, g\rangle|\\
&<&\frac{1}{K}(1-\varepsilon).
\end{eqnarray*}
By the arbitrary of $M,$ we have
\[\Big(\sum_{n=1}^{\infty} |\langle h,f_n\rangle|^p\Big)^{1/p}\leq
\frac{1}{K}(1-\varepsilon)<\frac{1}{K}=\frac{1}{K}|\langle
h,f\rangle|\leq\frac{1}{K}\|h\|,\] which leads to a contradiction.

For necessity, let $\{f_n\}$ be a $(C_q)$-system with constant $K>0$
in $X$. For every $h\in X^*$ and $\varepsilon>0$, there exists an $f\in
X$, $\|f\|=1$, and $|\langle h,f \rangle|=\|h\|$. Choose a linear
combination $g=\sum a_nf_n$ such that $\|f-g\|<\varepsilon$ and
$$\big(\sum|a_n|^q\big)^{1/q}\le K\|f\|=K.$$ Then
\begin{eqnarray*}
\|h\|&=&|\langle h,f \rangle|\\
&\le& |\langle h,f-g \rangle|+|\langle h,g \rangle| \\
&\le& \varepsilon\|h\|+\sum |a_n||\langle h,f_n\rangle|\\
&\le&\varepsilon\|h\|+\Big(\sum |a_n|^q\Big)^{1/q}\Big(\sum|\langle h,f_n\rangle|^p\Big)^{1/p}\\
&\le&\varepsilon\|h\|+K\Big(\sum|\langle h,f_n\rangle|^p\Big)^{1/p}.
\end{eqnarray*}
That is, $$\frac{1-\varepsilon}{K}\|h\|\le\Big(\sum|\langle
h,f_n\rangle|^p\Big)^{1/p}, \quad \forall \, h\in X^*, \ \forall \,
\varepsilon>0.$$ Since $\varepsilon$ is arbitrarily small, take
$\varepsilon\rightarrow0$, we complete the proof.
\end{proof}

The following result is elementary but very useful.
\begin{lemma}\label{lem:4} Let $1< p<\infty$, $f\in L^p(\R^d)$, and $\Gamma$ be a sequence in $\R^d$.
If $\Gamma$ is relatively uniformly separated, then for all cubes
$Q_h(x),$ for any $x\in\R^d$ and $h>0$, we have

(i) $\sum_{\gamma\in\Gamma} \|\chi_{Q_h(x)}T_\gamma f\|^p_p<\infty.$
\ \ \ \ \ (ii) $\sum_{\gamma\in\Gamma} \|\chi_{Q_h(x)}T_\gamma
f\|^p_p\rightarrow0, \ \ \mbox{ as } \, h\rightarrow0.$
\end{lemma}
\begin{proof} (i) Since $\Gamma$ is relatively uniformly separated, it
is a disjoint finite union of $\delta_k$-separated sequences
$\Gamma_k$ for $\delta_k>0$ with $k=1,...,n$. Let
$\delta=\displaystyle{\min_{1\le k\le n}} \delta_k>0$ be the
relatively separated constant and choose
$0<\varepsilon<\delta/\sqrt{d}$. Because any cube $Q_h(x)$ is
bounded, it must be contained in $Q_{2N\varepsilon}$ for some
$N\in\N$. Thus, it is enough to prove that $\sum_{\gamma\in\Gamma}
\|\chi_{Q_{2N\varepsilon}}T_\gamma f\|^p_p<\infty$ for all $N\in\N$.
For any $x\in\R^d$ and $h>0$, let
$$Q^+_h(x)=x+\prod_{j=1}^d[0,h)=\prod_{j=1}^d[x_j,x_j+h).$$
Then
\begin{eqnarray*}
\sum_{\gamma\in\Gamma} \|\chi_{Q_{2N\varepsilon}}T_\gamma
f\|^p_p&=&\sum_{k=1}^n\sum_{\gamma\in\Gamma_k}
\|\chi_{Q_{2N\varepsilon}}T_\gamma f\|^p_p\\
&=&\sum_{k=1}^n\sum_{\gamma\in\Gamma_k}\sum_{a\in Q_{2N}\cap\,\Z^d}
\|\chi_{Q^+_\varepsilon(\varepsilon a)}T_\gamma f\|^p_p\\
&=&\sum_{k=1}^n\sum_{a\in Q_{2N}\cap\,\Z^d}\sum_{\gamma\in\Gamma_k}
\|\chi_{Q^+_\varepsilon(\varepsilon a)}T_\gamma f\|^p_p\\
&=&\sum_{k=1}^n\sum_{a\in Q_{2N}\cap\,\Z^d}\sum_{\gamma\in\Gamma_k}
\|\chi_{Q^+_\varepsilon(\varepsilon a)-\gamma} f\|^p_p\\
&=&\sum_{k=1}^n\sum_{a\in Q_{2N}\cap\,\Z^d}\sum_{\gamma\in\Gamma_k}
\int_{Q^+_\varepsilon(\varepsilon a)-\gamma} |f(x)|^p\, dx.
\end{eqnarray*}
Since $\mathrm{diam} \big(Q^+_\varepsilon(\varepsilon
a)\big)=\sqrt{d}\varepsilon<\delta$ for any $a\in Q_{2N}\cap\,\Z^d$,
we get
$$Q^+_\varepsilon(\varepsilon
a)-\gamma=Q^+_\varepsilon(\varepsilon a-\gamma)$$
are mutually
disjoint for $\gamma\in\Gamma_k$. Thus
\begin{eqnarray}\label{eq:11}
\sum_{\gamma\in\Gamma} \|\chi_{Q_{2N\varepsilon}}T_\gamma f\|^p_p
&=&\sum_{k=1}^n\sum_{a\in Q_{2N}\cap\,\Z^d}\sum_{\gamma\in\Gamma_k}
\int_{Q^+_\varepsilon(\varepsilon a)-\gamma} |f(x)|^p\, dx\\
&\le& \sum_{k=1}^n\sum_{a\in
Q_{2N}\cap\,\Z^d}\|f\|_p^p\nonumber\\
&=&n(2N)^d\|f\|_p^p\nonumber\\
&<&\infty.\nonumber
\end{eqnarray}

(ii) For each $k=1,...,n,$ $a\in Q_{2N}\cap\Z^d$ and fixed
$x\in\R^d,$ we have
\[\chi_{Q^+_\varepsilon(\varepsilon a)-\Gamma_k}|f(x)|^p\rightarrow0
\ \ \ \mbox{as} \ \ \ \ \varepsilon\rightarrow0,\] here
$Q^+_\varepsilon(\varepsilon a)-\Gamma_k=\bigcup_{\gamma\in\Gamma_k}
Q^+_\varepsilon(\varepsilon a)-\gamma.$ Since
\[\chi_{Q^+_\varepsilon(\varepsilon a)-\Gamma_k}|f(x)|^p\leq |f(x)|^p,\]
by the Lebesgue Dominated Convergence Theorem, it follows that
\begin{eqnarray*}
\lim_{\varepsilon\rightarrow0}\sum_{\gamma\in\Gamma_k}\int_{Q^+_\varepsilon(\varepsilon
a)-\gamma} |f(x)|^p\,dx
&=&\lim_{\varepsilon\rightarrow0}\int_{Q^+_\varepsilon(\varepsilon
a)-\Gamma_k} |f(x)|^p\, dx\\
&=&\lim_{\varepsilon\rightarrow0\R^d}\int_{\R^d} \chi_{Q^+_\varepsilon(\varepsilon a)-\Gamma_k}|f(x)|^p\, dx\\
&=&0.
\end{eqnarray*}
Thus by (\ref{eq:11}),
\begin{eqnarray*}
\lim_{\varepsilon\rightarrow0}\sum_{\gamma\in\Gamma}
\|\chi_{Q_{2N\varepsilon}}T_\gamma f\|^p_p
&=& \lim_{\varepsilon\rightarrow0}\sum_{k=1}^n\sum_{a\in
Q_{2N}\cap\,\Z^d}\sum_{\gamma\in\Gamma_k}
\int_{Q^+_\varepsilon(\varepsilon a)-\gamma} |f(x)|^p\,dx\\
&=& \sum_{k=1}^n\sum_{a\in
Q_{2N}\cap\,\Z^d}\lim_{\varepsilon\rightarrow0}\sum_{\gamma\in\Gamma_k}
\int_{Q^+_\varepsilon(\varepsilon a)-\gamma} |f(x)|^p\, dx\\
&=&0.
\end{eqnarray*}
Thus, we obtain that if $\Gamma$ is relatively uniformly separated,
then for any $N\in\mathbb{N},$
\[\lim_{\varepsilon\rightarrow0}\sum_{\gamma\in\Gamma}
\|\chi_{Q_{2N\varepsilon}}T_\gamma f\|^p_p=0.\]
Since for all
$x\in\R^d$, the translation $\Gamma-x=\{\gamma-x:\gamma\in\Gamma\}$
of $\Gamma$ is relatively uniformly separated, then
\begin{eqnarray*}
&&\lim_{h\rightarrow0}\sum_{\gamma\in\Gamma} \|\chi_{Q_h(x)}T_\gamma
f\|^p_p=\lim_{h\rightarrow0}\sum_{\gamma\in\Gamma}
\|\chi_{Q_h}T_{\gamma-x}f\|^p_p=\lim_{h\rightarrow0}\sum_{\gamma\in\Gamma-x}
\|\chi_{Q_h}T_{\gamma} f\|^p_p=0.
\end{eqnarray*}
Now the conclusion follows.
\end{proof}

Now we prove our main result.

\begin{theorem}\label{th:1}
Let $1< p,q<\infty$ with $1/p+1/q=1$ and $n,d\in\N$. For each
$k=1,...,n,$ choose a nonzero function $f_k\in L^p(\R^d)$ and an
arbitrary sequence $\Gamma_k\subset\R^d$. Let $\Gamma$ be the
disjoint union of $\Gamma_1,...,\Gamma_n$.
\begin{enumerate}
\item[$(i)$] If for some $1< p'<\infty$, $\bigcup_{k=1}^n T_p(f_k,\Gamma_k)$ is a
$p'$-Bessel sequence, then $$D^+(\Gamma)<\infty.$$
\item[$(ii)$] If\, $\bigcup_{k=1}^n T_p(f_k,\Gamma_k)$ is a $(C_q)$-system, then $D^+(\Gamma)=\infty.$
\end{enumerate}

In particular, there is no $p'$-Bessel $(C_q)$-system in
$L^p(\R^d)$ of the form $\bigcup_{k=1}^n T_p(f_k,\Gamma_k)$.
\end{theorem}
\begin{proof}(i) Suppose that,
for some $1<p'<\infty$, $\bigcup_{k=1}^n T_p(f_k,\Gamma_k)$ is a
$p'$-Bessel sequence. It is equivalent to that each
$T_p(f_k,\Gamma_k)$ is a $p'$-Bessel sequence for $L^q(\R^d)$. Then,
by Proposition \ref{pp:1}, each $\Gamma_k$ is relatively uniformly
separated. By Lemma \ref{lemma:1}, $\Gamma_k$ has finite upper
Beurling density for each $1\le k\le n,$  i.e.
$D^+(\Gamma_k)<+\infty.$ Then by definition we have
\begin{eqnarray*}
\nu_\Gamma^+(h)&=&\sup_{x\in\R^d}\#(\Gamma\cap Q_h(x))\\
&=&\sup_{x\in\R^d}\#(\cup_{k=1}^n(\Gamma_k\cap Q_h(x)))\nonumber\\
&\le&\sum_{k=1}^n\sup_{x\in\R^d}\#(\Gamma_k\cap Q_h(x))\\
&=&\sum_{k=1}^n\nu_{\Gamma_k}^+(h).
\end{eqnarray*}
It follows that
\begin{eqnarray}\label{eq:1}
D^+(\Gamma)&=&\limsup_{h\rightarrow\infty}\frac{\nu_\Gamma^+(h)}{h^d}\nonumber\\
&\le& \limsup_{h\rightarrow\infty}\frac{\sum_{k=1}^n\nu_{\Gamma_k}^+(h)}{h^d}\nonumber\\
&\leq&\sum_{k=1}^n\limsup_{h\rightarrow\infty}\frac{\nu_{\Gamma_k}^+(h)}{h^d}\nonumber\\
&=&\sum_{k=1}^n D^+(\Gamma_k)\nonumber\\
&<&+\infty.
\end{eqnarray}
Thus, $\Gamma$ has finite upper Beurling density.

(ii) Since $\Gamma$ is the disjoint union of sequences $\Gamma_k$,
then, by formula (\ref{eq:1}), we have $ D^+(\Gamma)<\infty$ if and
only if $ D^+(\Gamma_k)<\infty$ for each $k=1,...,n.$ Assume that
$\Gamma$ has finite upper Beurling density. By Lemma \ref{lemma:1},
we know that $\Gamma_k$ is relatively uniformly separated. Now
consider the cube $Q_{2h}=\prod_{j=1}^d[-h,h)$ for $h>0$. Then
\begin{eqnarray*}
\sum_{k=1}^n \sum_{\gamma\in\Gamma_k} |\langle
\chi_{Q_{2h}},T_\gamma f_k \rangle|^p &=&\sum_{k=1}^n
\sum_{\gamma\in\Gamma_k} |\langle
\chi_{Q_{2h}},\chi_{Q_{2h}}T_\gamma f_k \rangle|^p\\
&\le&\sum_{k=1}^n \sum_{\gamma\in\Gamma_k}
\|\chi_{Q_{2h}}\|^p_q\|\chi_{Q_{2h}}T_\gamma f_k\|_p^p\\
&\le&\|\chi_{Q_{2h}}\|^p_q\sum_{k=1}^n
\sum_{\gamma\in\Gamma_k}\|\chi_{Q_{2h}}T_\gamma f_k\|_p^p.
\end{eqnarray*}
By Lemma \ref{lem:4}, we have for each $ k=1,...,n,$
$$\sum_{\gamma\in\Gamma_k}\|\chi_{Q_{2h}}T_\gamma f_k\|_p^p\to 0 \, \mbox{ as } h\to0
.$$ Thus, by Lemma \ref{lem:5}, it is easy to see that $\bigcup_k
T(f_k,\Gamma_k)$ is not a $(C_q)$-system. Thus, we complete the
proof.
\end{proof}

\begin{remark}
(i) The result due to Christensen, Deng and Heil \cite{CDH} is a special
case of Theorem \ref{th:1} for $p=p'=2$.

(ii) As a consequence of
Theorem \ref{th:1}, for no function $g\in L^p(\R^d)$ and no
constants $a,b>0$, $p'>1$ can a collection of functions of the form
$\{T_{na}E_{mb}g\}_{n\in\Z,m=1,...,M}$ be a $p'$-Bessel
$(C_q)$-system in $L^p(\R^d)$, where the modulation operator
$E_{mb}$ on $L^p(\R^d)$ is defined by $$(E_{mb}f)(x)=e^{2\pi i mb
x}f(x).$$ However, Hilbert frames of the infinite type
$\{T_{na}E_{mb}g\}_{m,n\in\Z}$ exist in $L^2(\R)$ (every Hilbert
frame is a Bessel $(C_2)$-system). For more information on Gabor
frames and density theorems, please see \cite{Ch2, He}.
\end{remark}

%
%
%

\section{Nonexistence of unconditional bases of
translates in $L^p(\R^d)$}
In this section, we will prove that
there doesn't exist any unconditional basis of the form $\bigcup_{k=1}^n
T_p(f_k,\Gamma_k)$ in $L^p(\R^d)$ for $1<p\le 2$. We use standard
Banach space notations as may be found in \cite{JL,LT}. Background
material on bases, unconditional bases and such can be found there.
For the benefit of those less familiar with these notions we recall
some definitions and facts.

A biorthogonal system is a sequence
$\{x_n,f_n\}\subset X\times X^*$ where $f_n(x_m)=\delta_{nm}.$

$\{x_n\}\subset X$ is a (Schauder) basis for $X$ if for all $x\in X,$
there exists a unique sequence of scalars $\{a_n\}$ so that
$x=\sum_{n=1}^\infty a_n x_n.$ This is equivalent to saying that all
$x_n\neq 0$, $\overline{\mathrm{span}}\{x_n\}=X$ and for some
$K<\infty$, all $m<l$ in $\N$ and all scalars $\{a_n\}_{n=1}^l$,
$$\left\|\sum_{n=1}^m a_nx_n\right\|\le K\left\|\sum_{n=1}^l a_n x_n\right\|.$$ The smallest
such $K$ is the basis constant of $\{x_n\}$.

$\{x_n\}$ is an
unconditional basis for $X$ if for all $x\in X,$ there exists a
unique sequence of scalars $\{a_n\}$ so that $x=\sum_{n=1}^\infty
a_n x_n$ and the convergence is unconditional. i.e.
$x=\sum_{n=1}^\infty a_{\pi(n)}x_{\pi(n)}$ for all permutations
$\pi$ of $\N$.

If $\{x_n\}$ is an unconditional basis for the Banach
space $X$ and $\theta=\{\theta_n\}_{n=1}^\infty$ is a sequence of
$\pm 1$'s, define $S_\theta:X\to X$ by $S_\theta(\sum\alpha_n
x_n)=\sum\theta_n\alpha_n x_n.$ The supremum over all such
$\|S_\theta\|$ is finite, and is called the unconditional constant
of the basis \cite{JL}.

The following lemma is easy to prove, which we leave to interested
readers.
\begin{lemma}\label{lem:3}
Let $X$ be a separable reflexive Banach space with
$\{x_n,f_n\}\subset X\times X^*.$ Assume that
$\{x_n,f_n\}$ is a biorthogonal system, that is, $\langle
x_n, f_m\rangle=\delta_{nm}$ for $n,m\in\N.$ Then $\{x_n\}$
is a seminormalized unconditional basis of $X$ if and only if
$\{f_n\}$ is a seminormalized unconditional basis of $X^*$.
\end{lemma}

Recall the following known inequalities in $L^p$-space \cite{AO}.
For $1<p<\infty$, there exist constants $A_p$, $B_p>0$ such that, if
$\{f_k\}_{k=1}^{\infty}$ is a normalized $C$-unconditional basic
sequence in $L^p(\R^d)$, then
\begin{equation}\label{eq:3}
(C A_p)^{-1}\big(\sum_{k=1}^{\infty}
|a_k|^2\big)^{1/2}\le\big\|\sum_{k=1}^{\infty}  a_k f_k\big\|_p \le
C \big(\sum_{k=1}^{\infty}|a_k|^p\big)^{1/p}, \, \mbox{ if } \,
1<p\le 2,\end{equation}
\begin{equation}\label{eq:4}
C^{-1}\big(\sum_{k=1}^{\infty}|a_k|^p\big)^{1/p}\le\big\|\sum_{k=1}^{\infty}
a_k f_k\big\|_p \le C B_p \big(\sum_{k=1}^{\infty}
|a_k|^2\big)^{1/2}, \, \mbox{ if } \, 2\le p<\infty.\end{equation}
\begin{proposition} \label{th:6}
Given $p,q\in(1,\infty)$ with $1/p+1/q=1$. Then
\begin{enumerate}
\item[$(i)$] If $1<p\le 2$, then every seminormalized unconditional
basis of $L^p(\R^d)$ is a $q$-Bessel $(C_2)$-system.
\item[$(ii)$] If $2\le p<\infty$, then every seminormalized unconditional
basis of $L^p(\R^d)$ is a Bessel $(C_p)$-system.
\end{enumerate}
\end{proposition}
\begin{proof}
Let $\{f_i\}$ be a seminormalized unconditional basis of
$L^p(\R^d)$, and $\{\tilde{f}_i\}\subset L^q(\R^d)$ be the
biorthogonal functionals of $\{f_i\}$. Then, by Lemma \ref{lem:3},
$\{\tilde{f}_i\}$ is a seminormalized $C$-unconditional basis of
$L^q(\R^d)$. Let $C_1=\inf\|\tilde{f}_i\|_p$ and $C_2=\sup\|\tilde{f}_i\|_q$. 

We first prove (i). Since $1<p\le 2$, we have $2\le q<\infty.$ By
inequality (\ref{eq:4}), for all $\tilde{f}\in L^q(\R^d)$,
\begin{eqnarray*}
\big(\sum |\langle \tilde{f},f_i\rangle|^q\big)^{1/q}&=&
\Big(\sum \frac{1}{\|\tilde{f}_i\|_q^q}|\langle
\tilde{f},\|\tilde{f}_i\|_q f_i \rangle|^q\Big)^{1/q}\\
&\le& \frac{1}{C_1}\big(\sum |\langle \tilde{f},\|\tilde{f}_i\|_q f_i \rangle|^q\big)^{1/q}\\
&\le& \frac{C}{C_1}\Big\|\sum \langle \tilde{f},\|\tilde{f}_i\|_q
f_i\rangle \frac{\tilde{f}_i}{\|\tilde{f}_i\|_q}\Big\|_q\\
&=&\frac{C}{C_1}\big\|\sum \langle \tilde{f}, f_i\rangle \tilde{f}_i\big\|_q\\
&=&\frac{C}{C_1}\|\tilde{f}\|_q.
\end{eqnarray*}
Moreover, for the lower 2-frame bound, we have
\begin{eqnarray*}
\big(\sum |\langle \tilde{f},f_i\rangle|^2\big)^{1/2}&=&
\Big(\sum \frac{1}{\|\tilde{f}_i\|_q^2}|\langle
\tilde{f},\|\tilde{f}_i\|_q f_i \rangle|^2\Big)^{1/2}\\
&\ge& \frac{1}{C_2}\big(\sum |\langle \tilde{f},\|\tilde{f}_i\|_q f_i \rangle|^2)^{1/2}\\
&\ge& \frac{1}{B_qCC_2}\Big\|\sum \langle
\tilde{f},\|\tilde{f}_i\|_q f_i\rangle
\frac{\tilde{f}_i}{\|\tilde{f}_i\|_q}\Big\|_q\\
&=&\frac{1}{B_qCC_2}\big\|\sum \langle \tilde{f}, f_i\rangle \tilde{f}_i\big\|_q\\
&=&\frac{1}{B_qCC_2}\|\tilde{f}\|_q^2.
\end{eqnarray*}

Now we prove (ii). Similarly, by inequality (\ref{eq:3}), for all
$\tilde{f}\in L^q(\R^d)$, we get that
$$\big(\sum |\langle \tilde{f},f_i\rangle|^2\big)^{1/2}\le \frac{CA_q}{C_1}\|\tilde{f}\|_q.$$
For the lower $q$-frame bound, we have
$$
\big(\sum |\langle
\tilde{f},f_i\rangle|^q\big)^{1/q}\ge\frac{1}{CC_2}\|\tilde{f}\|_q.
$$
Thus, we complete the proof.
\end{proof}

The following is the main result in this section.
\begin{theorem} Let $1< p\le 2$ and $n,d\in\N$. For each
$k=1,...,n,$ choose a nonzero function $f_k\in L^p(\R^d)$ and an
arbitrary sequence $\Gamma_k\subset\R^d$. Let $\Gamma$ be the
disjoint union of $\Gamma_1,...,\Gamma_n$. Then $\bigcup_{k=1}^n
T_p(f_k,\Gamma_k)$ can at most be an unconditional basis for a
proper subspace of $L^p(\R^d)$.

Equivalently, there is no unconditional basis of $L^p(\R^d)$ of the
form $\bigcup_{k=1}^n T_p(f_k,\Gamma_k)$.
\end{theorem}
\begin{proof}
If $\bigcup_{k=1}^n T_p(f_k,\Gamma_k)$ is an unconditional basis of
$L^p(\R^d)$, then, by Proposition \ref{th:6}, it is a $q$-Bessel
$(C_2)$-system for $L^p(\R^d)$. Since $1<p\le 2$, $2\le q<\infty$
with $1/p+1/q=1$, we have $(\sum |a_n|^q)^{1/q}\le (\sum
|a_n|^2)^{1/2}$. Then, by (\ref{eq:5}) in Definition \ref{def:1},
$\bigcup_{k=1}^n T_p(f_k,\Gamma_k)$ is a $q$-Bessel $(C_q)$-system
for the whole $L^p(\R^d)$. It leads to a contradiction by Theorem
\ref{th:1}.
\end{proof}

{\bf Acknowledgement.} This work was partially done when the second
author visited Departments of Mathematics in Texas A\&M University
and the University of Texas at Austin. The second author would like
to thank David Larson, Edward Odell and Thomas Schlumprecht for the
invitation and great help.

Both authors also expresses their appreciation to Qiyu Sun and
Wenchang Sun for very helpful comments.

\bibliographystyle{amsplain}

\begin{thebibliography}{99}


\bibitem{AO} D. Alspach, E. Odell, \textit{$L_p$ spaces}, Hanbook of the Geometry of Banach Spaces,
Vol. 1, W.B. Johnson and J. Lindenstrauss, eds, Elsevier, Amsterdam
(2001), 123--160.

\bibitem{AOl} A. Atzmon and A. Olevskii, \textit{Completeness of integer
translates in function spaces on $\R$}, J. Approx. Theory
\textbf{87} (1996), 291--327.

\bibitem{AST} A. Aldroubi, Q. Sun and W. Tang, \textit{$p$-frames and
shift invariant subspaces of $L^p$}, J. Fourier Anal. Appl.
\textbf{7} (2001), no. 1, 1--22.

%
%


\bibitem{Ca} P.G. Casazza, \textit{The art of frame theory}, Taiwanese J.
Math. \textbf{4} (2000), no. 2, 129--201.

\bibitem{CCS} P. Casazza, O. Christensen and D.T. Stoeva, \textit{Frame expansions in separable Banach spaces},
J. Math. Anal. Appl. \textbf{307} (2005), 710--723.



\bibitem{Ch1} O. Christensen, \textit{Frames contaning a Riesz basis and approximation of the frame coefficients},
J. Math. Anal. Appl. \textbf{199} (1995), 256--270.

\bibitem{Ch2} O. Christensen, \textit{An introduction to frames and Riesz
bases}, Birkh$\ddot{\mathrm{a}}$user, 2003.

\bibitem{CS} O. Christensen and D.T. Stoeva,
\textit{$p$-frames in separable Banach spaces}, Adv. Comp. Math.
\textbf{18} (2003), 117--126.

\bibitem{DS} R.J. Duffin and A.C. Schaeffer, \textit{A class of nonharmonic Fourier series},
Tran. Amer. Math. Soc. \textbf{72} (1952), 341--366.

\bibitem{ER} G. Edgar and J. Rosenblatt, \textit{Difference equations
over locally compact abelian groups}, Tran. Amer. Math. Soc.
\textbf{253} (1979), 273--289.

\bibitem{CDH} O. Christensen, B. Deng and C. Heil, \textit{Density of Gabor frames},
Appl. Comput. Harmon. Anal. \textbf{7} (1999), 292--304.



\bibitem{He} Christopher Heil, \textit{History and evolution of the
density theorems for Gabor frames}, J. Fourier Anal. Appl.
\textbf{13} (2007), 113--166.

\bibitem{JL} W.B. Johnson, J. Lindenstrauss, \textit{Basic concepts in the
geometry of Banach spaces}, Hanbook of the Geometry of Banach
Spaces, Vol. 1, W.B. Johnson and J. Lindenstrauss, eds, Elsevier,
Amsterdam (2001), 1--84.

\bibitem{LT} J. Lindenstrauss and L. Tzafriri, \textit{Classical Banach spaces I: Sequence Spaces},
Springer-Verlag, Berlin (1977).
%
%
%

\bibitem{NO1} S. Nitzan and A. Olevskii,
\textit{Sparse exponential systems: completeness with estimates},
Israel J. Math. \textbf{158} (2007), 205--215.

\bibitem{NO2} S. Nitzan and A. Olevskii,
\textit{Quasi-frames of translates}, C. R. Math. Acad. Sci. Paris
\textbf{347} (2009), 739--742.

\bibitem{NO3} S. Nitzan and J. Olsen,
\textit{From exact systems to Riesz bases in the Balian-Low
Theorem}, J. Fourier Anal. Appl. \textbf{17} (2011), 567--603.

\bibitem{OSSZ} E. Odell, B. Sari, Th. Schlumprecht and B. Zheng,
\textit{Systems formed by translates of one element in $L_p(\R)$},
Trans. Amer. Math. Soc. \textbf{363} (2011), 6505--6529.

\bibitem{Ol} A. Olevskii, \textit{Completeness in $L_2(\R)$ of almost
integer translates}, C. R. Acad. Sci. Paris \textbf{324} (1997),
987--991.

\bibitem{OZ} T.E. Olson and R.A. Zalik, \textit{Nonexistence of a Riesz
basis of translates, in: Approximation Theory}, Leture Notes in Pure
and Applied Math., 138, Dekker, New York (1992), 401--408.



\bibitem{Wi} N. Wiener, \textit{The Fourier integral and certain of its
applications}, Cambridge University Press, Cambridge, 1933 reprint:
Dover, New York, 1958.



%
%
%
%
%


\end{thebibliography}

\end{document}